\documentclass[letterpaper,english]{article}
\usepackage[T1]{fontenc}
\usepackage[latin9]{inputenc}
\usepackage{amsmath}
\usepackage{amsthm}

\makeatletter

\theoremstyle{plain}
\newtheorem{thm}{\protect\theoremname}
\theoremstyle{plain}
\newtheorem{prop}{\protect\propositionname}
\theoremstyle{plain}
\newtheorem{lem}{\protect\lemmaname}

\usepackage{babel}

\providecommand{\lemmaname}{Lemma}
\providecommand{\propositionname}{Proposition}
\providecommand{\theoremname}{Theorem}

\makeatother

\usepackage{babel}
\providecommand{\lemmaname}{Lemma}
\providecommand{\propositionname}{Proposition}
\providecommand{\theoremname}{Theorem}

\begin{document}
\title{Faulhaber's Formula, Odd Bernoulli Numbers, and the Method of Partial
Sums\thanks{This work is licensed under the CC BY 4.0, a Creative Commons Attribution
License.}}
\author{Ryan Zielinski\\
 ryan\_zielinski@fastmail.com}
\date{4 August 2022}

\maketitle

\begin{abstract}
\noindent Let ``Faulhaber's formula'' refer to an expression for
the sum of powers of integers written with terms in n(n+1)/2. Initially,
the author used Faulhaber's formula to explain why odd Bernoulli numbers
are equal to zero. Next, Cereceda gave alternate proofs of that result
and then proved the converse, if odd Bernoulli numbers are equal to
zero then we can derive Faulhaber's formula. Here, the original author
will give a new proof of the converse using the method of partial
sums and mathematical induction. 
\end{abstract}

\section{Motivation}

If we knew nothing of the history of the problem and tried to discover
for ourselves a general expression for 
\[
\sum_{k=1}^{n}k^{m}=1^{m}+2^{m}+\cdots+n^{m},
\]
where $n,m$ are positive integers, we might notice there appear to
be two ways to write such sums. For example, 
\begin{equation}
\begin{aligned}\sum k= & \frac{n(n+1)}{2}=\frac{1}{2}n^{2}+\frac{1}{2}n,\\
\sum k^{2}= & \frac{2n+1}{3}\cdot\frac{n(n+1)}{2}=\frac{1}{3}n^{3}+\frac{1}{2}n^{2}+\frac{1}{6}n,\\
\sum k^{3}= & \left(\frac{n(n+1)}{2}\right)^{2}=\frac{1}{4}n^{4}+\frac{1}{2}n^{3}+\frac{1}{4}n^{2}.
\end{aligned}
\label{eq:1}
\end{equation}
\pagebreak{}

The next two cases are 
\begin{equation}
\begin{aligned}\sum k^{4}= & \frac{1}{5}\left[6\cdot\frac{n(n+1)}{2}-1\right]\cdot\sum k^{2}=\frac{1}{5}n^{5}+\frac{1}{2}n^{4}+\frac{1}{3}n^{3}-\frac{1}{30}n,\\
\sum k^{5}= & \frac{1}{3}\left[4\cdot\frac{n(n+1)}{2}-1\right]\left(\sum k\right)^{2}=\frac{1}{6}n^{6}+\frac{1}{2}n^{5}+\frac{5}{12}n^{4}-\frac{1}{12}n^{2}.
\end{aligned}
\label{eq:2}
\end{equation}
(When no confusion will arise, we will abbreviate $\sum_{k=1}^{n}k^{m}$
by $\sum k^{m}$.)

We can write each sum using terms of $\frac{n(n+1)}{2}$ or $n$.
Of course, if we have the former then we always can expand it into
the latter. Do we always have the former?

At a later point, reading up on the matter we would learn that writing
an expression for $\sum k^{m}$ using terms in $n$ is associated
with the name of Jakob Bernoulli (1654-1705), and writing the same
expression using terms in $\frac{n(n+1)}{2}$ is associated with that
of Johann Faulhaber (1580-1635). Bernoulli's contribution has long
overshadowed Faulhaber's, but now we know the two are linked inextricably.

\section{Background}

In order to write an expression for $\sum k^{m}$ in $n$, we introduce
the Bernoulli numbers. Set $B_{0}=1$ and define $B_{n}$ by 
\[
\sum_{k=0}^{n}\binom{n+1}{k}B_{k}=0,
\]
where $n\geq1$. Then we can write 
\begin{equation}
\sum k^{m}=\frac{1}{m+1}\sum_{j=0}^{m}\left(-1\right)^{j}\binom{m+1}{j}B_{j}n^{m+1-j}.\label{eq:3}
\end{equation}
For a sum in $\frac{n(n+1)}{2}$, for even powers we have 
\begin{equation}
\sum k^{2m}=\left[c_{0}\left(\frac{n(n+1)}{2}\right)^{m-1}+\cdots+c_{m-2}\cdot\frac{n(n+1)}{2}+c_{m-1}\right]\cdot\sum k^{2},\label{eq:4}
\end{equation}
and for odd powers we have 
\begin{equation}
\sum k^{2m+1}=\left[a_{0}\left(\frac{n(n+1)}{2}\right)^{m-1}+\cdots+a_{m-2}\cdot\frac{n(n+1)}{2}+a_{m-1}\right]\left(\sum k\right)^{2},\label{eq:5}
\end{equation}
where the $c_{i},a_{i}$ are rational numbers and $m\geq1$. We will
refer to these two expressions as ``Faulhaber's formula.''

Regarding earlier work on the topic, Edwards \cite{key-6,key-7} took
a matrix-based approach and looked for recurrence relations amongst
all such sums \eqref{eq:4} and \eqref{eq:5}. The coefficients $c_{i},a_{i}$
then followed as entries in the inverses of such matrices. Gessel
and Viennot \cite[section 12]{key-8} investigated sums of powers
and alternating sums of powers. Explicit expressions for both sets
of coefficients were derived and their combinatorial properties were
discussed. Knuth \cite{key-12} looked at a lot of material. Perhaps
the most important part was examining Faulhaber's original work and
placing it into a modern perspective.

Now we begin with the present contribution, which aimed to unravel
just what allows for writing $\sum k^{m}$ with terms in $\frac{n(n+1)}{2}$.
But, we started in the opposite direction.

If we look at the expressions in \eqref{eq:1} and \eqref{eq:2},
we notice a few powers of $n$ are missing. The reason is because
odd Bernoulli numbers are equal to zero: $B_{1}=-\frac{1}{2}$, but
for all $m\geq1$, $B_{2m+1}=0$. When we write such sums using \eqref{eq:3},
the powers of $n$ which have odd Bernoulli numbers for coefficients
drop out.

There are established ways to prove such a result (\cite[Chapters 1-2]{key-13}).
The new insight of Zielinski in \cite{key-15} was that, with the
different expressions for $\sum k^{m}$, we already have enough information
to justify such an outcome.

If we write $\sum k^{2m+1}$ in the two forms of \eqref{eq:3} and
\eqref{eq:5}, the coefficients for the terms in $n$ must agree.
Since \eqref{eq:5} contains a factor of $\left(\sum k\right)^{2}=\frac{1}{4}n^{4}+\frac{1}{2}n^{3}+\frac{1}{4}n^{2}$,
there is no term of $n$. That means the last term of \eqref{eq:3},
$-B_{2m+1}\cdot n$, must be equal to zero. In other words, $B_{2m+1}=0$
for all $m\geq1$.\footnote{If we write \eqref{eq:4} in a manner analogous to (2.5) of \cite{key-15},
we can give the same type of argument for even powers.}

To prove the converse, Cereceda \cite{key-5} chose a different line
of attack and introduced Bernoulli polynomials: 
\begin{equation}
B_{m}\left(x\right)=\sum_{j=0}^{m}\binom{m}{j}B_{j}x^{m-j},\label{eq:6}
\end{equation}
where $x$ is a real variable and $B_{m}\left(0\right)=B_{m}$. This
allows an expression for the sum of powers to be written as 
\begin{equation}
\sum k^{m}=\frac{1}{m+1}\left(B_{m+1}(n+1)-B_{m+1}\right).\label{eq:7-1}
\end{equation}

In the new context, the property $B_{2m+1}=0$ is related to Bernoulli
polynomials being evaluated at $x=\frac{1}{2}$. The notion of symmetry
allows such polynomials to be rewritten with terms of $\left(x-\frac{1}{2}\right)$,
which can be rewritten once more with terms of $\frac{x(x-1)}{2}$.
Once this is done, \eqref{eq:7-1} leads immediately to Faulhaber's
formula. (Again, \cite{key-5} contains alternate proofs of the main
result of \cite{key-15}.)

Together, the papers lead to a surprising revelation, one which has
been a long time in the making. Denote $\sum_{k=1}^{n}k^{m}$ by $S_{m}$.
Then we have 
\begin{thm}[Zielinski \cite{key-15} and Cereceda \cite{key-5}]
For positive integers $m$, 
\[
B_{2m+1}=0\iff\begin{cases}
S_{2m}= & S_{2}\cdot F_{2m}\left(S_{1}\right),\\
S_{2m+1}= & S_{1}^{2}\cdot F_{2m+1}\left(S_{1}\right),
\end{cases}
\]
where $F_{2m}\left(S_{1}\right)$ and $F_{2m+1}\left(S_{1}\right)$
are polynomials in $S_{1}=\frac{n(n+1)}{2}$. 
\end{thm}
In this paper we will give a different proof of the converse. Our
approach will be based on what commonly is referred to as the method
of partial sums: 
\[
\sum_{k=1}^{n}k^{m+1}=\left(n+1\right)\sum_{k=1}^{n}k^{m}-\sum_{k=1}^{n}\sum_{l=1}^{k}l^{m}.
\]
We state it in a pointed fashion to illustrate that it serves as a
generator for sums of powers (\cite[Section 2]{key-14}).\footnote{This exact form of the method of partial sums is said to go back 1,000
years to ibn al-Haytham (965-1039). Katz \cite{key-11} and Boudreaux
\cite{key-2} discuss some of the history.} In this context, when we use \eqref{eq:3} to write expressions for
$\sum\sum l^{m}$, the property $B_{2m+1}=0$ will cause the bulk
of the terms to be of even or odd parity like $\sum k^{m+1}$.

\section{Main Result}

We start with a proof of our method of partial sums. 
\begin{prop}
\label{for-positive-integers}For integers $n,m$, where $n\geq1$
and $m\geq0$, 
\[
\sum_{k=1}^{n}k^{m+1}+\sum_{k=1}^{n}\sum_{l=1}^{k}l^{m}=\left(n+1\right)\sum_{k=1}^{n}k^{m}.
\]
\end{prop}
\begin{proof}
A way to prove $\frac{n(n+1)}{2}=\sum_{k=1}^{n}k$ is to write 
\[
n(n+1)=2\left(1+2+\cdots+n\right)
\]
and then to interpret the left side as a rectangle of area $n(n+1)$
and the right side as two pieces of $\left(1+2+\cdots+n\right)$ squares.
We will do something analogous in this case.

Let us write $n+1$ rows of 
\[
\begin{array}{c}
1^{m}+2^{m}+\cdots+\left(n-1\right)^{m}+n^{m}\\
1^{m}+2^{m}+\cdots+\left(n-1\right)^{m}+n^{m}\\
\vdots\\
1^{m}+2^{m}+\cdots+\left(n-1\right)^{m}+n^{m}\\
1^{m}+2^{m}+\cdots+\left(n-1\right)^{m}+n^{m}
\end{array}
\]
and then divide the sum into 
\begin{align*}
 &  &  &  & 1^{m}+2^{m}+\cdots+\left(n-1\right)^{m}+n^{m}\\
 & 1^{m} &  &  & 2^{m}+\cdots+\left(n-1\right)^{m}+n^{m}\\
 & 1^{m}+2^{m} &  & + & \vdots\\
 & \vdots &  &  & \left(n-1\right)^{m}+n^{m}\\
 & 1^{m}+2^{m}+\cdots+\left(n-1\right)^{m} &  &  & n^{m}\\
 & 1^{m}+2^{m}+\cdots+\left(n-1\right)^{m}+n^{m}
\end{align*}
This gives us 
\[
\left(n+1\right)\sum_{k=1}^{n}k^{m}=\sum_{k=1}^{n}\sum_{l=1}^{k}l^{m}+\sum_{k=1}^{n}k^{m+1}.
\]
\end{proof}
Before we prove the converse we state, without proof, a lemma which
points out critical, intermediate relationships. 
\begin{lem}
For positive integers $n$, 
\[
\begin{aligned}\left(n+\frac{1}{2}\right)\left(\sum k\right)^{2}= & \frac{3}{2}\cdot\frac{n(n+1)}{2}\cdot\sum k^{2},\\
\left(n+\frac{1}{2}\right)\sum k^{2}= & \left(\frac{4}{3}\cdot\frac{n(n+1)}{2}+\frac{1}{6}\right)\sum k.
\end{aligned}
\]
\end{lem}
A consequence of the second relationship is that $\left(\sum k^{2}\right)^{2}$
can be rewritten using only terms of $\sum k$, which then implies,
through \eqref{eq:4}, that $\left(\sum k^{2m}\right)^{2}$ can be
rewritten in $\sum k$ as well (\cite[Section 3]{key-1}). We have

\begin{equation}
\left(\sum k^{2}\right)^{2}=G\left(S_{1}\right)\;\textrm{and}\;\left(\sum k^{2m}\right)^{2}=H\left(S_{1}\right),\label{eq:7}
\end{equation}
where $G\left(S_{1}\right)$ and $H\left(S_{1}\right)$ are polynomials
in $S_{1}=\frac{n(n+1)}{2}$.

\medskip{}

Now we give a new proof of the converse. 
\begin{prop}
If odd Bernoulli numbers are equal to zero, we can derive Faulhaber's
formula. 
\end{prop}
\begin{proof}
We proceed by mathematical induction. Expression \ref{eq:3}, with
$B_{3}=0,$ allows us to write 
\[
\begin{aligned}\sum_{k=1}^{n}k^{2}= & \frac{1}{3}n^{3}+\frac{1}{2}n^{2}+\frac{1}{6}n=\frac{2n+1}{3}\cdot\frac{n(n+1)}{2}=1\cdot\sum k^{2},\\
\sum_{k=1}^{n}k^{3}= & \frac{1}{4}n^{4}+\frac{1}{2}n^{3}+\frac{1}{4}n^{2}-0\cdot n=\left(\frac{n(n+1)}{2}\right)^{2}=1\cdot\left(\sum k\right)^{2}.
\end{aligned}
\]
We will assume \eqref{eq:4} and \eqref{eq:5} are true for all $1,2,...,m$
and then establish the case of $m+1$.

For $m+1$, Proposition \ref{for-positive-integers} tells us 
\begin{equation}
\sum_{k=1}^{n}k^{2m+2}=\left(n+1\right)\sum_{k=1}^{n}k^{2m+1}-\sum_{k=1}^{n}\sum_{l=1}^{k}l^{2m+1}.\label{eq:9}
\end{equation}
We want to write $\sum l^{2m+1}$ in $k$. Expression \eqref{eq:3}
tells us 
\[
\sum_{l=1}^{k}l^{2m+1}=\frac{1}{2m+2}\sum_{j=0}^{2m+1}\left(-1\right)^{j}\binom{2m+2}{j}B_{j}k^{2m+2-j}.
\]
Since we are assuming $B_{3}=B_{5}=\cdots=0$, we can rewrite the
double sum as 
\begin{eqnarray*}
\sum_{k=1}^{n}\sum_{l=1}^{k}l^{2m+1} & = & \frac{1}{2m+2}\sum_{k=1}^{n}\left(1\cdot B_{0}k^{2m+2}-\binom{2m+2}{1}B_{1}k^{2m+1}\right.\\
 &  & \left.+\binom{2m+2}{2}B_{2}k^{2m}+\cdots+\binom{2m+2}{2m}B_{2m}k^{2}\right)\\
 & = & \frac{1}{2m+2}\sum_{k=1}^{n}k^{2m+2}+\frac{1}{2}\sum_{k=1}^{n}k^{2m+1}+b_{2m}\sum_{k=1}^{n}k^{2m}+\cdots+b_{2}\sum_{k=1}^{n}k^{2},
\end{eqnarray*}
where $b_{2m},\ldots,b_{2}$ are rational numbers which do not interest
us. Now we can rewrite \eqref{eq:9} as 
\begin{eqnarray}
\sum k^{2m+2} & = & \left(n+1\right)\sum k^{2m+1}-\frac{1}{2m+2}\sum k^{2m+2}-\frac{1}{2}\sum k^{2m+1}\nonumber \\
 &  & -\left(b_{2m}\sum k^{2m}+\cdots+b_{2}\sum k^{2}\right).\label{eq:10}
\end{eqnarray}
By the inductive hypothesis for $\sum k^{2m}$, we can rewrite the
sum in the parentheses as 
\[
b_{2m}\sum k^{2}\cdot F_{2m}+\cdots+b_{2}\sum k^{2}\cdot F_{2},
\]
where $F_{2m},\ldots,F_{2}$ are polynomials in $\frac{n(n+1)}{2}$.
Together, this is just $F\cdot\sum k^{2}$ for another such polynomial
$F$. Now \eqref{eq:10} becomes 
\begin{equation}
\frac{2m+3}{2m+2}\sum k^{2m+2}=\left(n+\frac{1}{2}\right)\sum k^{2m+1}-F\cdot\sum k^{2}.\label{eq:11}
\end{equation}

For the next step of the proof, first we invoke the inductive hypothesis
for $\sum k^{2m+1}.$ This allows us to rewrite the right side of
\eqref{eq:11} as 
\begin{equation}
\left(n+\frac{1}{2}\right)\left(\sum k\right)^{2}G_{2m+1}-F\cdot\sum k^{2},\label{eq:12}
\end{equation}
where $G_{2m+1}$ is a polynomial in $\frac{n(n+1)}{2}$. Then we
use the lemma to rewrite the left side of \eqref{eq:12} as 
\[
\frac{3}{2}\cdot\frac{n(n+1)}{2}\cdot\sum k^{2}\cdot G_{2m+1}=G\cdot\sum k^{2},
\]
where $G$ is another polynomial in $\frac{n(n+1)}{2}$. The final
form of \eqref{eq:11} becomes 
\begin{eqnarray*}
\sum k^{2m+2} & = & \frac{2m+2}{2m+3}\left(G\cdot\sum k^{2}-F\cdot\sum k^{2}\right)=H\cdot\sum k^{2},
\end{eqnarray*}
where $H$ is a polynomial in $\frac{n(n+1)}{2}$.

\medskip{}
 The proof for $\sum k^{2m+3}$ proceeds along the same lines. We
only wish to point out an important difference when rewriting the
double sum using Bernoulli numbers. Starting with 
\[
\sum_{k=1}^{n}k^{2m+3}=\left(n+1\right)\sum_{k=1}^{n}k^{2m+2}-\sum_{k=1}^{n}\sum_{l=1}^{k}l^{2m+2},
\]
the expression analogous to \eqref{eq:11} will be 
\begin{equation}
\frac{2m+4}{2m+3}\sum k^{2m+3}=\left(n+\frac{1}{2}\right)\sum k^{2m+2}-B_{2m+2}\sum k-F\cdot\left(\sum k\right)^{2}.\label{eq:13}
\end{equation}
We need to eliminate the term of $-B_{2m+2}\sum k$, which we do as
follows.

The coefficient of $B_{2m+2}$ comes out of writing $\sum l^{2m+2}$
according to \eqref{eq:3}. If we write the same expression using
\eqref{eq:4}, which we just established, we get 
\[
\sum k^{2m+2}=\left[G+c_{m}\right]\sum k^{2}=\left[G+c_{m}\right]\cdot\frac{2n^{3}+3n^{2}+n}{6},
\]
where $G$ is a polynomial in $\frac{n(n+1)}{2}$, of degree of at
least one, and $c_{m}$ is a rational number. The coefficient for
the term of $n$ is $\frac{c_{m}}{6}$. Since both coefficients must
agree, we have $\frac{c_{m}}{6}=B_{2m+2}$.

When we invoke the lemma we get 
\begin{eqnarray*}
\left(n+\frac{1}{2}\right)\left[G+c_{m}\right]\sum k^{2} & = & \left[G+c_{m}\right]\left(\frac{4}{3}\cdot\frac{n(n+1)}{2}+\frac{1}{6}\right)\sum k\\
 & = & \left[G+c_{m}\right]\left(\frac{4}{3}\left(\sum k\right)^{2}+\frac{1}{6}\sum k\right)\\
 & = & \left[G+c_{m}\right]\cdot\frac{4}{3}\left(\sum k\right)^{2}+G\cdot\frac{1}{6}\sum k+\frac{c_{m}}{6}\sum k.
\end{eqnarray*}
Since the polynomial $G$ does not have a constant term, we can simplify
the right side to 
\begin{equation}
G^{'}\cdot\left(\sum k\right)^{2}+\frac{c_{m}}{6}\sum k,\label{eq:14}
\end{equation}
where $G^{'}$ is another polynomial in $\frac{n(n+1)}{2}$. Now the
term of $\frac{c_{m}}{6}\sum k$ cancels with that of $-B_{2m+2}\sum k$,
and \eqref{eq:13} and \eqref{eq:14} become 
\[
\frac{2m+4}{2m+3}\sum k^{2m+3}=G^{'}\cdot\left(\sum k\right)^{2}-F\cdot\left(\sum k\right)^{2},
\]
from which the desired result follows. 
\end{proof}

\section{Concluding Remarks}

\subsection{Sums in $\left(n+\frac{1}{2}\right)$}

In the preceding pages there have been hints that there is a third
way to write an expression for the sum of powers of integers, with
terms in \textbf{$\left(n+\frac{1}{2}\right)$}. This in fact is true,
and can be approached in a number of different ways (see the work
of Beardon \cite{key-1}, Burrows and Talbot \cite{key-3}, Cereceda
\cite{key-4}, and Hersh \cite{key-10}). Since we have developed
the theory of writing such sums with terms in $\frac{n(n+1)}{2}$,
we will proceed in that direction.

Let us start with even powers. We have the identity 
\begin{equation}
\left(n+\frac{1}{2}\right)^{2}=2\cdot\frac{n(n+1)}{2}+\frac{1}{4},\label{eq:15}
\end{equation}
which allows us to write 
\begin{equation}
\begin{aligned}\sum k & =\frac{n(n+1)}{2}=\frac{1}{2}\left(n+\frac{1}{2}\right)^{2}-\frac{1}{8},\\
\sum k^{2} & =\frac{2n+1}{3}\cdot\frac{n(n+1)}{2}=\left(n+\frac{1}{2}\right)\left(\frac{1}{3}\left(n+\frac{1}{2}\right)^{2}-\frac{1}{12}\right).
\end{aligned}
\label{eq:16}
\end{equation}
Using the shorthand $N=n+\frac{1}{2}$ we can rewrite \eqref{eq:4}
as 
\begin{eqnarray}
\sum k^{2m} & = & \left[c_{0}\left(\frac{1}{2}N^{2}-\frac{1}{8}\right)^{m-1}+\cdots+c_{m-1}\right]\cdot N\left(\frac{1}{3}N^{2}-\frac{1}{12}\right)\nonumber \\
 & = & \left[c_{0}^{'}N^{2\left(m-1\right)}+\cdots+c_{m-2}^{'}N^{2}+c_{m-1}^{'}\right]\cdot N\left(\frac{1}{3}N^{2}-\frac{1}{12}\right)\nonumber \\
 & = & \left[c_{0}^{''}N^{2\left(m-1\right)+2}+\cdots+c_{m-2}^{''}N^{2+2}+c_{m-1}^{''}N^{2}+c_{m}^{''}\right]\cdot N\nonumber \\
 & = & d_{0}\left(n+\frac{1}{2}\right)^{2m+1}+\cdots+d_{m-1}\left(n+\frac{1}{2}\right)^{3}+d_{m}\left(n+\frac{1}{2}\right),\label{eq:17}
\end{eqnarray}
which is an odd polynomial in $\left(n+\frac{1}{2}\right)$ with rational
coefficients $d_{i}$. For odd powers we can rewrite \eqref{eq:5}
as 
\begin{eqnarray}
\sum k^{2m+1} & = & \left[a_{0}\left(\frac{1}{2}N^{2}-\frac{1}{8}\right)^{m-1}+\cdots+a_{m-1}\right]\left(\frac{1}{2}N^{2}-\frac{1}{8}\right)^{2}\nonumber \\
 & = & \left[a_{0}^{'}N^{2\left(m-1\right)}+\cdots+a_{m-2}^{'}N^{2}+a_{m-1}^{'}\right]\left(\frac{1}{4}N^{4}-\frac{1}{8}N^{2}+\frac{1}{64}\right)\nonumber \\
 & = & e_{0}\left(n+\frac{1}{2}\right)^{2m+2}+\cdots+e_{m}\left(n+\frac{1}{2}\right)^{2}+e_{m+1},\label{eq:18}
\end{eqnarray}
which is an even polynomial in $\left(n+\frac{1}{2}\right)$ with
rational coefficients $e_{i}$.

It is possible to proceed in the opposite direction: start with \eqref{eq:17}
and \eqref{eq:18} and then derive \eqref{eq:4} and \eqref{eq:5}.
However, some care needs to be taken with the constant terms. Concerning
the coefficients $d_{i},e_{i}$, we can find explicit values in terms
of Bernoulli numbers: 
\begin{equation}
\begin{aligned}d_{i} & =\frac{1}{2\left(m-i\right)+1}\binom{2m}{2i}B_{2i}\left(\frac{1}{2}\right),\\
e_{i} & =\frac{1}{2\left(m-i\right)+2}\binom{2m+1}{2i}B_{2i}\left(\frac{1}{2}\right),
\end{aligned}
\label{eq:19}
\end{equation}
where $0\leq i\leq m$. We evaluate the Bernoulli polynomial $B_{r}\left(x\right)$
using 
\begin{equation}
B_{r}\left(\frac{1}{2}\right)=\left(2^{1-r}-1\right)B_{r}.\label{eq:20}
\end{equation}
For the lone coefficient of $e_{m+1}$ we have 
\begin{equation}
e_{m+1}=-\sum_{i=0}^{m}\frac{e_{i}}{4^{m-i+1}}.\label{eq:21}
\end{equation}
(These results were derived in a slightly different form in \cite{key-4}.)

\subsection{Partial sums for Bernoulli polynomials}

Since the expressions for sums of powers of integers have the form
of polynomials, it is common to interpret them as such. Previously
we used the notation $S_{m}=\sum_{k=1}^{n}k^{m}$. We introduce $S_{m}\left(x\right)$
to designate the polynomial of degree $m+1$ in the real variable
$x$ that is obtained by replacing $n$ by $x$ in \eqref{eq:3}.

In \cite[Theorem 2.2]{key-9}, He and Ricci derived the following
expression: 
\begin{equation}
B_{m}\left(x\right)=\left(x-\frac{1}{2}\right)B_{m-1}\left(x\right)-\frac{1}{m}\sum_{r=0}^{m-2}\binom{m}{r}B_{m-r}B_{r}\left(x\right),\label{eq:22}
\end{equation}
where $m\geq1$, and $B_{m}\left(x\right)$ refers to the Bernoulli
polynomials defined in \eqref{eq:6}. It followed as a corollary to
more general results on Appell polynomials which were derived using
differential operators. Here we will use the simpler approach of partial
sums to derive an expression for $S_{m}\left(x\right)$ which is equivalent
to \eqref{eq:22}.

Starting with Proposition \ref{for-positive-integers}, we can arrive
at 
\[
S_{2m+2}=\frac{1}{2m+3}\left(\left(2m+2\right)\left(n+\frac{1}{2}\right)S_{2m+1}-\sum_{r=1}^{2m}\binom{2m+2}{r}B_{2m+2-r}S_{r}\right),
\]
which is analogous to \eqref{eq:11}. (Note: we do not assume $B_{2m+1}=0$
and we write \eqref{eq:3} without the alternating signs.) Likewise,
an expression analogous to \eqref{eq:13} is 
\[
S_{2m+3}=\frac{1}{2m+4}\left(\left(2m+3\right)\left(n+\frac{1}{2}\right)S_{2m+2}-\sum_{r=1}^{2m+1}\binom{2m+3}{r}B_{2m+3-r}S_{r}\right).
\]
Putting them together, we get 
\[
S_{m}=\frac{1}{m+1}\left(m\left(n+\frac{1}{2}\right)S_{m-1}-\sum_{r=1}^{m-2}\binom{m}{r}B_{m-r}S_{r}\right),
\]
where $m\geq2$, with the understanding that for $m=2$, the sum on
the right side is equal to zero. Generalizing to polynomials in $x$,
we arrive at the final result of 
\begin{equation}
S_{m}\left(x\right)=\frac{1}{m+1}\left(m\left(x+\frac{1}{2}\right)S_{m-1}\left(x\right)-\sum_{r=1}^{m-2}\binom{m}{r}B_{m-r}S_{r}\left(x\right)\right).\label{eq:23}
\end{equation}
Even though we will not give the details of the proof, we point out
that \eqref{eq:22} and \eqref{eq:23} can be shown to be equivalent
to each other.

\section*{Acknowledgements}

The author thanks the referee for a careful reading of the paper,
for suggesting the topics in the concluding section, and for pointing
out additional references, in particular, the paper by He and Ricci.

\end{document}